\documentclass[preprint,12pt]{elsarticle}

\usepackage{amssymb}
\usepackage{epsfig}
\usepackage{amsmath}
\usepackage{mathrsfs}
\usepackage{lineno}
\usepackage{url}

\newtheorem{thm}{Theorem}
\newtheorem{lem}[thm]{Lemma}
\newproof{proof}{Proof}

\journal{Elsevier}

\begin{document}

\begin{frontmatter}



\title{Analytical evaluation and asymptotic evaluation of Dawson's integral and related functions in mathematical physics}


\author{V. Nijimbere}
\ead{victornijimbere@gmail.com}
\address{School of Mathematics and Statistics, Carleton University, Ottawa, Ontario, Canada}


\begin{abstract}
Dawson's integral and related functions in mathematical physics that include the complex error function (Faddeeva's integral), Fried-Conte (plasma dispersion) function, (Jackson) function, Fresnel function and Gordeyev's integral are analytically evaluated in terms of the confluent hypergeometric function.
And hence, the asymptotic expansions of these functions on the complex plane $\mathbb{C}$ are derived using the asymptotic expansion of the confluent hypergeometric function. 
\end{abstract}


\end{frontmatter}

\section{Introduction}\label{sec:1}

Let us consider the first-order initial value problem,
\begin{equation}
D^\prime+2zD=1, D(0)=0.
\label{eqn:1.1}
\end{equation}
Its solution given by the definite integral
\begin{equation}
\text{daw}\hspace{.12cm} z=D(z)=e^{-z^2}\int\limits_0^z e^{\eta^2}d\eta
\label{eqn:1.2}
\end{equation}
is known as Dawson's integral \cite{AS,McC,Sc,We}. Dawson's integral is related to several important functions (in integral form) in mathematical physics that include Faddeeva's integral (also know as the complex error function or Kramp function) \cite{FT1,FT2,PW,Mi,We}
\begin{equation}
w(z)=e^{-z^2}\left[1+\frac{2i}{\sqrt{\pi}}e^{z^2}\text{daw}\hspace{.12cm} z\right]=e^{-z^2}\left(1+\frac{2i}{\sqrt{\pi}}\int\limits_0^z e^{\eta^2}d\eta\right),
\label{eqn:1.3}
\end{equation}
Fried-Conte function (or plasma dispersion function) \cite{BT,FC}
\begin{equation}
Z(z)=i\sqrt{\pi}w(z)=i\sqrt{\pi}e^{-z^2}\left[1+\frac{2i}{\sqrt{\pi}}e^{z^2}\text{daw}\hspace{.12cm} z\right]=i\sqrt{\pi}e^{-z^2}\left(1+\frac{2i}{\sqrt{\pi}}\int\limits_0^z e^{\eta^2}d\eta\right),
\label{eqn:1.4}
\end{equation}
(Jackson) function \cite{Ja}
\begin{align}
G(z)&=1+zZ(z)=1+i\sqrt{\pi}zw(z)\nonumber \\ &
=1+i\sqrt{\pi}ze^{-z^2}\left[1+\frac{2i}{\sqrt{\pi}}e^{z^2}\text{daw}\hspace{.12cm} z\right]\nonumber \\ &
=1+i\sqrt{\pi}ze^{-z^2}\left(1+\frac{2i}{\sqrt{\pi}}\int\limits_0^z e^{\eta^2}d\eta\right),
\label{eqn:1.5}
\end{align}
and Fresnel functions $C(x)$ and $S(x)$ \cite{AS} defined by the relation
\begin{equation}
\frac{e^{{i\pi}{z^2}}}{\sqrt{i\pi}}\text{daw}\left({\sqrt{i\pi}}{z}\right)=\int\limits_0^z e^{i\pi\eta^2}d\eta=C(x)+i S(x),
\label{eqn:1.6}
\end{equation}
where
\begin{equation}
C(x)=\int\limits_0^z \cos{(\pi\eta^2)}d\eta \hspace{0.2cm} \text{and} \hspace{0.2cm} S(x)=\int\limits_0^z \sin{(\pi\eta^2)}d\eta.
\label{eqn:1.7}
\end{equation}

There is also Gordeyev's integral \cite{Go} which is related to Dawson's integral via the plasma dispersion (Fried-Conte) function $Z$, and is given
\begin{align}
G_\nu(\omega,\lambda)&=\omega\int\limits_0^\infty e^{i\omega t-\lambda(1-\cos{t})-\nu t^2/2} dt \nonumber \\&
=\frac{-i\omega}{\sqrt{2\nu}}e^{-\lambda}\sum\limits_{n=-\infty}^{\infty}I_n(\lambda)Z\left(\frac{\omega-n}{\sqrt{2\nu}}\right),\hspace{.12cm} \text{Re}(\nu)>0,
\label{eqn:1.8}
\end{align}
where $I_n$ is the Bessel function of the first kind \cite{AS}, the real part of $\omega$ is the wave frequency of an electrostatic wave propagating in a hot magnetized plasma, and $\lambda$ and $\nu$ are respectively the squares of the perpendicular and parallel components of the wave vector \cite{Pa}.

Another function related to Dawson's integral was defined by Sitenko \cite{Si}, and only takes real arguments. It is given by
\begin{equation}
\varphi(x)=2x\hspace{0.1cm} \text{daw}\hspace{.12cm} x=2xe^{-x^2}\int\limits_0^x e^{\eta^2}d\eta.
\label{eqn:1.9}
\end{equation}
In that case, for real argument $x$, Jackson function $G(x)$, given by (\ref{eqn:1.5}), takes the form
\begin{equation}
G(x)=1-\varphi(x)+i\sqrt{\pi}xe^{-x^2}.
\label{eqn:1.10}
\end{equation}

These functions find their applications in astronomy, celestial mechanics, optical physics, plasma physics, planetary atmosphere, radiophysics, spectroscopy and so on \cite{BT,BTG,FC,Go,Ja,Mi,Si,St}. Therefore, it is important to adequately evaluate them.
Having computed Dawson's integral or  Faddeeva's integral, it is then straightforward to compute the other related integrals or functions such as Fried-Conte function, Jackson function, Fresnel integral and Gordeyev function.

First, it is important to point out that Dawson's integral and Faddeeva's integral are non-elementary integrals. Being non-elementary means that they cannot neither be expressed in terms of elementary functions such polynomials of finite degree, exponentials and logarithms, nor in terms of mathematical expressions obtained by performing finite algebraic combinations involving elementary functions \cite{MZ,Ni,Ro}. For this reason, it is not possible to evaluate analytically these integrals in closed form or, in other words, in terms of elementary functions \cite{AQ2,MZ,Ni,Ro}.
To this end, intensive works have mainly focused on numerical approximations \cite{AQ1,AQ2,CG,CP,FT1,FT2,Ga,McC,McK,PW,We,Za}. But numerical integrations do have drawbacks, they become very expensive and inaccurate very quickly as $z$ becomes large or for some values of $z$, see for example the recent work by Abrarov and Quine \cite{AQ2}.

None of the above integrals can be evaluated analytically in close form, or in terms of elementary functions, as pointed out by Abrarov and Quine \cite{AQ2} of course, but one can express them in terms of a special function, the confluent hypergeometric function $_{1}F_1$ \cite{AS,ND}. Noting that Nijimbere \cite{Ni} (Theorem 1) has evaluated the non-elementary integral $\int_a^b e^{\lambda x^\alpha} dx, \alpha\ge2$ for any constant $\lambda$ in terms of the confluent hypergeometric function $_{1}F_1$, as a first objective of this paper, we use the results in Nijimbere \cite{Ni} to obtain an analytical expression for Dawson's integral in terms of the confluent hypergeometric function $_{1}F_1$. And hence, we write Faddeeva's integral and the above related integrals in terms of $_{1}F_1$.

On the other hand, the confluent hypergeometric function is an entire function on the whole complex plane $\mathbb{C}$, and its properties are well known \cite{AS,ND}. For instance, its asymptotic expansion is given in \cite{AS,ND}. Therefore, as a second objective of this work, the asymptotic expansion of the confluent hypergeometric function is used to obtain the asymptotic expansions of the above functions (integrals) on the complex plane $\mathbb{C}$.

\section{Evaluation of Dawson's integral and related function in terms of $_{1}F_1$}\label{sec:2}
In this section, Dawson's integral is evaluated in terms of the confluent hypergeometric $_{1}F_1$, and relations (\ref{eqn:1.3}), (\ref{eqn:1.4}), (\ref{eqn:1.5}), (\ref{eqn:1.6}) and (\ref{eqn:1.8}) are used to express Faddeeva, Fried-Conte, Jackson, Fresnel and Gordeyev integrals respectively in terms of $_{1}F_1$.
\subsection{Evaluation of Dawson's integral in terms of $_{1}F_1$}\label{subsec:2.1}
We use Theorem 1 in Nijimbere \cite{Ni} to express (\ref{eqn:1.2}) in terms of the confluent hypergeometric functions $_{1}F_1$, and then obtain
\begin{equation}
\text{daw}\hspace{.12cm} z=D(z)=e^{-z^2}\int\limits_0^z e^{\eta^2}d\eta= ze^{-z^2}\hspace{.05cm}{}_1F_1\left(\frac{1}{2};\frac{3}{2};z^2\right).
\label{eqn:2.11}
\end{equation}
One may also solve (\ref{eqn:1.1}) and obtain \cite{Ni} the general solution
\begin{equation}
D(z)=e^{-z^2}\left(\int\limits_0^z e^{\eta^2}d\eta+C\right)= ze^{-z^2}\hspace{.05cm}{}_1F_1\left(\frac{1}{2};\frac{3}{2};z^2\right)+C e^{-z^2}.
\label{eqn:2.12}
\end{equation}
Therefore, applying the initial condition $D(0)=0$ gives (\ref{eqn:2.11}).
\subsection{Evaluation of related functions in terms of $_{1}F_1$}\label{subsec:2.2}
We now can evaluate the other functions related to Dawson's integral (see section \ref{sec:1}) in terms of the confluent hypergeometric function. Using (\ref{eqn:1.3}) Faddeeva's integral is now given by
\begin{equation}
w(z)=e^{-z^2}\left(1+\frac{2i}{\sqrt{\pi}}\int\limits_0^z e^{\eta^2}d\eta\right)= e^{-z^2}\left[1+\frac{2i z}{\sqrt{\pi}}\hspace{.05cm}{}_1F_1\left(\frac{1}{2};\frac{3}{2};z^2\right)\right].
\label{eqn:2.13}
\end{equation}
Using (\ref{eqn:1.4}), the plasma dispersion function $Z(z)$ also called Fried-Conte function is given by
\begin{equation}
Z(z)=i\sqrt{\pi} w(z)= i\sqrt{\pi} e^{-z^2}\left[1+\frac{2i z}{\sqrt{\pi}}\hspace{.05cm}{}_1F_1\left(\frac{1}{2};\frac{3}{2};z^2\right)\right].
\label{eqn:2.14}
\end{equation}
Using (\ref{eqn:1.5}), the (Jackson) function, denoted by $G(z)$ is given by
\begin{equation}
G(z)=1+zZ(z)=1+i\sqrt{\pi}zw(z)=1+i\sqrt{\pi} ze^{-z^2}\left[1+\frac{2i z}{\sqrt{\pi}}\hspace{.05cm}{}_1F_1\left(\frac{1}{2};\frac{3}{2};z^2\right)\right].
\label{eqn:2.15}
\end{equation}
Using (\ref{eqn:1.6}) (one may also use Proposition 1 in \cite{Ni}) gives
\begin{equation}
\int\limits_0^z e^{i\pi\eta^2}d\eta=\frac{e^{{i\pi}{z^2}}}{\sqrt{i\pi}}\text{daw}\left({\sqrt{i\pi}}{z}\right)
=z\hspace{.05cm}{}_1F_1\left(\frac{1}{2};\frac{3}{2};i\pi z^2\right),
\label{eqn:2.16}
\end{equation}
which is equivalent to formula 7.3.25 in \cite{AS}. Now using formula (53) in Theorem 2 and formula (57) in Theorem 3 in \cite{Ni} (or using formulas (50) and (51) in \cite{Ni}), we obtain
\begin{equation}
C(x)=\int\limits_0^z \cos{(\pi\eta^2)}d\eta=z\hspace{.05cm}{}_1F_2\left(\frac{1}{4};\frac{1}{2},\frac{5}{4};-\frac{\pi z^2}{4}\right),
\label{eqn:2.17}
\end{equation}
and
\begin{equation}
S(x)=\int\limits_0^z \sin{(\pi\eta^2)}d\eta=\frac{\pi z^3}{3}\hspace{.05cm}{}_1F_2\left(\frac{3}{4};\frac{3}{2},\frac{7}{4};-\frac{\pi z^2}{4}\right).
\label{eqn:2.18}
\end{equation}
Setting $z=\frac{\omega-n}{\sqrt{2\nu}}$ in Fried-Conte function and substituting in (\ref{eqn:1.8}), Gordeyev's integral can now be written in terms of ${}_1F_1$ as
\begin{align}
G_\nu(\omega,\lambda)&= \frac{-i\omega}{\sqrt{2\nu}}e^{-\lambda}\sum\limits_{n=-\infty}^{\infty}I_n(\lambda)Z\left(\frac{\omega-n}{\sqrt{2\nu}}\right) \label{eqn:2.19}\\ &
=\frac{\omega e^{-\lambda}}{\sqrt{2\nu/\pi}}\sum\limits_{n=-\infty}^{\infty}e^{-\left(\frac{\omega-n}{\sqrt{2\nu}}\right)^2}I_n(\lambda)\left\{1+
\frac{2i(\omega-n)}{\sqrt{2\pi\nu}}\hspace{.05cm}{}_1F_1\left[\frac{1}{2};\frac{3}{2};\left(\frac{\omega-n}{\sqrt{2\nu}}\right)^2\right]\right\}.
\label{eqn:2.20}
\end{align}

\section{Asymptotic evaluation}\label{sec:3}
In this section, we derive the asymptotic expansion of Dawson's integral and related functions that include the complex error function, Fried-Conte function, Jackson function, Fresnel functions and Gordeyev's function, using the asymptotic expansion of the confluent hypergeometric ${}_1F_1$. The results are summarized in Theorem \ref{thm:1}, Theorem \ref{thm:2} and Theorem \ref{thm:3}.
\subsection{Asymptotic evaluation of Dawson's integral}\label{subsec:3.1}
\begin{lem}
For $|z|\gg1$,

\begin{multline}
z\hspace{.075cm}{}_1F_1\left(\frac{1}{\alpha};\frac{1}{\alpha}+1;z^\alpha\right)\\ \sim\begin{cases}
\Gamma\left(\frac{1}{\alpha}+1\right)e^{\pm i\frac{\pi}{\alpha}}\frac{z}{|z|}\left[1+\frac{\Gamma\left(\frac{1}{\alpha}+1\right)}{z^\alpha}+O\left(\frac{1}{z^{2\alpha}}\right)\right]+\frac{e^{z^\alpha}}{\alpha z^{\alpha-1}}\left[1+\frac{\Gamma\left(2-\frac{1}{\alpha}\right)}{z^\alpha}+O\left(\frac{1}{z^{2\alpha}}\right)\right], & \text{if $\alpha$ is even} \\
\Gamma\left(\frac{1}{\alpha}+1\right)e^{\pm i\frac{\pi}{\alpha}}\left[1+\frac{\Gamma\left(\frac{1}{\alpha}+1\right)}{z^\alpha}+O\left(\frac{1}{z^{2\alpha}}\right)\right]+\frac{e^{z^\alpha}}{\alpha z^{\alpha-1}}\left[1+\frac{\Gamma\left(2-\frac{1}{\alpha}\right)}{z^\alpha}+O\left(\frac{1}{z^{2\alpha}}\right)\right], & \text{if $\alpha$ is odd} \\
\end{cases}\\
\label{eqn:3.21}
\end{multline}
where $\alpha$ is a constant, and the positive sign is taken if
\begin{equation}
-\frac{\pi}{2\alpha}+\frac{2k\pi}{\alpha} <\text{arg}(z)< \frac{3\pi}{2\alpha}+\frac{2k\pi}{\alpha},
\label{eqn:3.22}
\end{equation}
while the negative sign is taken if
\begin{equation}
-\frac{3\pi}{2\alpha}+\frac{2k\pi}{\alpha} <\text{arg}(z)< -\frac{\pi}{2\alpha}+\frac{2k\pi}{\alpha},
\label{eqn:3.23}
\end{equation}
$k=0,1,2,\cdot\cdot\cdot$.
\label{lm:1}
\end{lem}

\begin{proof}
To prove (\ref{eqn:3.21}), we use the asymptotic expansion of the confluent hypergeometric function valid for $|\xi|\gg 1$ (\cite{AS}, formula 13.5.1),
\begin{multline}
\frac{{}_1F_1\left(a;b;\xi\right)}{\Gamma(b)}=\frac{e^{\pm i\pi a}\xi^{-a}}{\Gamma(b-a)}\left\{\sum\limits_{n=0}^{R-1}\frac{(a)_n (1+a-b)_n}{n!}(-\xi)^{-n}+O(|\xi|^{-R})\right\}\\
+\frac{e^{\xi}\xi^{a-b}}{\Gamma(a)}\left\{\sum\limits_{n=0}^{S-1}\frac{(b-a)_n (1-a)_n}{n!}(\xi)^{-n}+O(|\xi|^{-S})\right\},
\label{eqn:3.24}
\end{multline}
where $a$ and $b$ are constants, and
the positive sign being taken if

\begin{equation}
-\frac{\pi}{2}<\text{arg}(\xi)<\frac{3\pi}{2}, \label{eqn:3.25}
\end{equation}
and the negative sign if

\begin{equation}
-\frac{3\pi}{2}<\text{arg}(\xi)\le-\frac{\pi}{2}.
\label{eqn:3.26}
\end{equation}

We now set $\xi=z^\alpha, a=\frac{1}{\alpha}$ and $b=\frac{1}{\alpha}+1$ in (\ref{eqn:3.23}), and obtain
\begin{multline}
\frac{{}_1F_1\left(\frac{1}{\alpha};\frac{1}{\alpha}+1; z^\alpha\right)}{\Gamma\left(\frac{1}{\alpha}+1\right)}=\frac{e^{i\pm\frac{\pi}{\alpha}}}{( z^\alpha)^{\frac{1}{\alpha}}}\left\{\sum\limits_{n=0}^{R-1}\frac{\left(\frac{1}{\alpha}\right)_n }{n!}(z^\alpha)^{-n}+O\left( z^\alpha\right)^{-R}\right\} \\
+\frac{1}{\Gamma\left(\frac{1}{\alpha}\right)}\frac{e^{z^\alpha}}{z^\alpha}\left\{\sum\limits_{n=0}^{S-1}\frac{(1)_n \left(1-\frac{1}{\alpha}\right)_n}{n!}(z^\alpha)^{-n}+O\left(z^\alpha\right)^{-S}\right\}.
\label{eqn:3.27}
\end{multline}

Then for $|z|\gg1$,
\begin{equation}
\frac{e^{\pm i\frac{\pi}{\alpha}}}{(z^\alpha)^{\frac{1}{\alpha}}}\left\{\sum\limits_{n=0}^{R-1}\frac{\left(\frac{1}{\alpha}\right)_n }{n!}(z^\alpha)^{-n}+O\left(z^\alpha\right)^{-R}\right\} \\ \sim\begin{cases}
\frac{e^{\pm i\frac{\pi}{\alpha}}}{|z|}\left[1+\frac{\Gamma\left(\frac{1}{\alpha}+1\right)}{z^\alpha}+O\left(\frac{1}{z^{2\alpha}}\right)\right], & \text{if $\alpha$ is even}  \\
\frac{e^{\pm i\frac{\pi}{\alpha}}}{z}\left[1+\frac{\Gamma\left(\frac{1}{\alpha}+1\right)}{z^\alpha}+O\left(\frac{1}{z^{2\alpha}}\right)\right], & \text{if $\alpha$ is odd} \\
\end{cases}
\label{eqn:3.28}
\end{equation}
while
\begin{equation}
\frac{1}{\Gamma\left(\frac{1}{\alpha}\right)}\frac{e^{z^\alpha}}{z^\alpha}\left\{\sum\limits_{n=0}^{S-1}\frac{(1)_n \left(\frac{1}{\alpha}\right)_n}{n!}(z^\alpha)^{-n}+O\left(z^\alpha\right)^{-S}\right\}\sim \frac{1}{\Gamma\left(\frac{1}{\alpha}\right)}\frac{e^{z^\alpha}}{z^\alpha}\left[1+\frac{\Gamma\left(2-\frac{1}{\alpha}\right)}{z^\alpha}+O\left(\frac{1}{z^{2\alpha}}\right)\right].
\label{eqn:3.29}
\end{equation}
Therefore for $|z|\gg1$,
\begin{multline}
{}_1F_1\left(\frac{1}{\alpha};\frac{1}{\alpha}+1;z^\alpha\right)\\ \sim\begin{cases}
\Gamma\left(\frac{1}{\alpha}+1\right)\frac{e^{\pm i\frac{\pi}{\alpha}}}{|z|}\left[1+\frac{\Gamma\left(\frac{1}{\alpha}+1\right)}{z^\alpha}+O\left(\frac{1}{z^{2\alpha}}\right)\right]+\frac{e^{z^\alpha}}{\alpha z^\alpha}\left[1+\frac{\Gamma\left(2-\frac{1}{\alpha}\right)}{z^\alpha}+O\left(\frac{1}{z^{2\alpha}}\right)\right], & \text{if $\alpha$ is even} \\
\Gamma\left(\frac{1}{\alpha}+1\right)\frac{e^{\pm i\frac{\pi}{\alpha}}}{z}\left[1+\frac{\Gamma\left(\frac{1}{\alpha}+1\right)}{z^\alpha}+O\left(\frac{1}{z^{2\alpha}}\right)\right]+\frac{e^{z^\alpha}}{\alpha z^\alpha}\left[1+\frac{\Gamma\left(2-\frac{1}{\alpha}\right)}{z^\alpha}+O\left(\frac{1}{z^{2\alpha}}\right)\right], & \text{if $\alpha$ is odd} \\
\end{cases}\\
\label{eqn:3.30}
\end{multline}
Hence for $|z|\gg1$,
\begin{multline}
z\hspace{.075cm}{}_1F_1\left(\frac{1}{\alpha};\frac{1}{\alpha}+1;z^\alpha\right)\\ \sim\begin{cases}
\Gamma\left(\frac{1}{\alpha}+1\right)e^{\pm i\frac{\pi}{\alpha}}\frac{z}{|z|}\left[1+\frac{\Gamma\left(\frac{1}{\alpha}+1\right)}{z^\alpha}+O\left(\frac{1}{z^{2\alpha}}\right)\right]+\frac{e^{z^\alpha}}{\alpha z^{\alpha-1}}\left[1+\frac{\Gamma\left(2-\frac{1}{\alpha}\right)}{z^\alpha}+O\left(\frac{1}{z^{2\alpha}}\right)\right], & \text{if $\alpha$ is even} \\
\Gamma\left(\frac{1}{\alpha}+1\right)e^{\pm i\frac{\pi}{\alpha}}\left[1+\frac{\Gamma\left(\frac{1}{\alpha}+1\right)}{z^\alpha}+O\left(\frac{1}{z^{2\alpha}}\right)\right]+\frac{e^{z^\alpha}}{\alpha z^{\alpha-1}}\left[1+\frac{\Gamma\left(2-\frac{1}{\alpha}\right)}{z^\alpha}+O\left(\frac{1}{z^{2\alpha}}\right)\right], & \text{if $\alpha$ is odd} \\
\end{cases}\\
\label{eqn:3.31}
\end{multline}

On the other hand, we observe that
$$ \xi=z^\alpha\Rightarrow z=\xi^{\frac{1}{\alpha}}=|\xi|^{\frac{1}{\alpha}}\left[e^{i\hspace{.1cm}\text{arg}(\xi)}\right]^{\frac{1}{\alpha}}=|\xi|^{\frac{1}{\alpha}}
e^{i\frac{\text{arg}(\xi)}{\alpha}}.$$
Therefore,  (\ref{eqn:3.25}) gives
$$-\frac{\pi}{2\alpha}+\frac{2k\pi}{\alpha} <\text{arg}(z)< \frac{3\pi}{2\alpha}+\frac{2k\pi}{\alpha},$$
which is exactly (\ref{eqn:3.22}), while (\ref{eqn:3.26}) gives
$$-\frac{3\pi}{2\alpha}+\frac{2k\pi}{\alpha} <\text{arg}(z)< -\frac{\pi}{2\alpha}+\frac{2k\pi}{\alpha},$$
which is exactly (\ref{eqn:3.23}), and where $k=0,1,2,\cdot\cdot\cdot.$
\end{proof}

\begin{thm} For $|z|\gg1$ if $-\frac{\pi}{4}+k\pi <\text{arg}(z)< \frac{3\pi}{4}+k\pi, k=0,1,2,\cdot\cdot\cdot$,
then Dawson's integral is asymptotically given by
\begin{equation}
\text{daw}\hspace{.15cm} z\sim i\frac{\sqrt{\pi}}{2}e^{-z^2}\left[1+\frac{\sqrt{\pi}}{2z^2}+O\left(\frac{1}{z^4}\right)\right]+\frac{1}{2z}+\frac{\sqrt{\pi}}{4z^3}
+O\left(\frac{1}{z^5}\right).
\label{eqn:3.32a}
\end{equation}
While, on the other hand, it is given by
\begin{equation}
\text{daw}\hspace{.15cm} z\sim-i\frac{\sqrt{\pi}}{2}e^{-z^2}\left[1+\frac{\sqrt{\pi}}{2z^2}+O\left(\frac{1}{z^4}\right)\right]+\frac{1}{2z}+\frac{\sqrt{\pi}}{4z^3}
+O\left(\frac{1}{z^5}\right)
\label{eqn:3.32b}
\end{equation}
if $-\frac{3\pi}{4}+k\pi <\text{arg}(z)< -\frac{\pi}{4}+k\pi,k=0,1,2,\cdot\cdot\cdot$.
\\\\Therefore, for $|z|\gg1$, if $-\frac{\pi}{4}+k\pi <\text{arg}(z)< \frac{3\pi}{4}+k\pi,k=0,1,2,\cdot\cdot\cdot$, then
\begin{enumerate}
\item Feddeeva's integral $w(z)$ is approximated  by
\begin{equation}
w(z)\sim
e^{-z^2}\left[-\frac{\sqrt{\pi}}{2z^2}+O\left(\frac{1}{z^4}\right)\right]+\frac{i}{\sqrt{\pi}}\left[\frac{1}{z}+\frac{\sqrt{\pi}}{2z^3}
+O\left(\frac{1}{z^5}\right)\right],
\label{eqn:3.33a}
\end{equation}
\item Fried-Conte (plasma dispersion) function $Z(z)$ by
\begin{equation}
Z(z)\sim  i e^{-z^2}\left[-\frac{\pi}{2z^2}+O\left(\frac{1}{z^4}\right)\right]-\frac{1}{2z}-\frac{\sqrt{\pi}}{4z^3}
+O\left(\frac{1}{z^5}\right)
\label{eqn:3.34a}
\end{equation}
\item and (Jackson) function $G(z)$ by
\begin{equation}
G(z)\sim ie^{-z^2}\left[-\frac{\pi}{2z}+O\left(\frac{1}{z^3}\right)\right]-\frac{\sqrt{\pi}}{2z^2}
+O\left(\frac{1}{z^4}\right).
\label{eqn:3.35a}
\end{equation}
\end{enumerate}

While on the other hand, if $-\frac{3\pi}{4}+k\pi <\text{arg}(z)< -\frac{\pi}{4}+k\pi$, then

\begin{enumerate}
\item Feddeeva's integral $w(z)$ is approximated  by
\begin{equation}
w(z)\sim
e^{-z^2}\left[2+\frac{\sqrt{\pi}}{2z^2}+O\left(\frac{1}{z^4}\right)\right]+\frac{i}{\sqrt{\pi}}\left[\frac{1}{z}+\frac{\sqrt{\pi}}{2z^3}
+O\left(\frac{1}{z^5}\right)\right],
\label{eqn:3.33b}
\end{equation}
\item Fried-Conte (plasma dispersion) function $Z(z)$ by
\begin{equation}
Z(z)\sim
i\sqrt{\pi}e^{-z^2}\left[2+\frac{\sqrt{\pi}}{2z^2}+O\left(\frac{1}{z^4}\right)\right]-\frac{1}{z}-\frac{\sqrt{\pi}}{2z^3}
+O\left(\frac{1}{z^5}\right),
\label{eqn:3.34b}
\end{equation}
\item and (Jackson) function $G(z)$ by
\begin{equation}
G(z)\sim
i\sqrt{\pi}e^{-z^2}\left[2z+\frac{\sqrt{\pi}}{2z}+O\left(\frac{1}{z^3}\right)\right]-\frac{\sqrt{\pi}}{2z^2}
+O\left(\frac{1}{z^4}\right)
\label{eqn:3.35b}
\end{equation}
\end{enumerate}
\label{thm:1}
\end{thm}

\begin{proof}
For $\alpha=2$, and having in mind that $\alpha=2$ is even, (\ref{eqn:3.21}) becomes
\begin{align}
z\hspace{.075cm}{}_1F_1\left(\frac{1}{2};\frac{3}{2};z^2\right)
&\sim\Gamma\left(\frac{3}{2}\right)e^{\pm i\frac{\pi}{2}}\frac{z}{|z|}\left[1+\frac{\Gamma\left(\frac{3}{2}\right)}{z^2}+O\left(\frac{1}{z^{4}}\right)\right]+\frac{e^{z^2}}{2 z}\left[1+\frac{\Gamma\left(\frac{3}{2}\right)}{z^2}+O\left(\frac{1}{z^{4}}\right)\right]\nonumber\\&
=\pm i\frac{\sqrt{\pi}}{2}\frac{z}{|z|}\left[1+\frac{\sqrt{\pi}}{2
z^2}+O\left(\frac{1}{z^{4}}\right)\right]+\frac{e^{z^2}}{2 z}\left[1+\frac{\sqrt{\pi}}{2
z^2}+O\left(\frac{1}{z^{4}}\right)\right],
\label{eqn:3.36}
\end{align}
where the positive sign is taken if

\begin{equation}
-\frac{\pi}{4}+k\pi <\text{arg}(z)< \frac{3\pi}{4}+k\pi,
\label{eqn:3.37}
\end{equation}
while the negative sign is taken if

\begin{equation}
-\frac{3\pi}{4}+k\pi <\text{arg}(z)< -\frac{\pi}{4}+k\pi,
\label{eqn:3.38}
\end{equation}
$k=0,1,2,\cdot\cdot\cdot$.

When substituting (\ref{eqn:3.36}) in (\ref{eqn:2.12}), the resulting equation together with (\ref{eqn:3.37}) and (\ref{eqn:3.38}) respectively gives
(\ref{eqn:3.32a}) if $-\frac{\pi}{4}+k\pi <\text{arg}(z)< \frac{3\pi}{4}+k\pi$ and (\ref{eqn:3.32b}) if $-\frac{3\pi}{4}+k\pi <\text{arg}(z)< -\frac{\pi}{4}+k\pi$. And hence, substituting  (\ref{eqn:3.32a}) and (\ref{eqn:3.32b}) in (\ref{eqn:2.13}), (\ref{eqn:2.14}) and (\ref{eqn:2.15}) gives respectively
(\ref{eqn:3.33a}), (\ref{eqn:3.34a}) and (\ref{eqn:3.35a}) if $-\frac{\pi}{4}+k\pi <\text{arg}(z)< \frac{3\pi}{4}+k\pi$ and (\ref{eqn:3.33b}), (\ref{eqn:3.34b}) and (\ref{eqn:3.35b}) if $-\frac{3\pi}{4}+k\pi <\text{arg}(z)< -\frac{\pi}{4}+k\pi$. This ends the proof.

\end{proof}
\subsection{Asymptotic evaluation of Fresnel functions}\label{subsec:3.2}
In this section, we use the asymptotic expansion of the confluent hypergeometric function ${}_1F_1$ to derive the asymptotic expansion of Fresnel's functions $C(z)$ and $S(z)$. And the results are described in Theorem \ref{thm:2}.
\begin{thm} For $|z|\gg1$,
Fresnel's functions $C(z)$ and $S(z)$ are asymptotically given by
\begin{equation}
C(z)\sim \begin{cases} \frac{\sqrt{2}}{4}-\sqrt{\frac{2}{\pi}}\frac{1}{8z^2}-\frac{\sin{(\pi z^2)}}{2\pi z}-\frac{\cos{(\pi z^2)}}{4(\pi)^{3/2} z^3}+O\left(\frac{1}{z^4}\right) , & \text{if} \hspace{.12cm} -\frac{\pi}{2}+k\pi <\text{arg}(z)< \frac{\pi}{2}+k\pi  \\
-\frac{\sqrt{2}}{4}+\sqrt{\frac{2}{\pi}}\frac{1}{8z^2}-\frac{\sin{(\pi z^2)}}{2\pi z}-\frac{\cos{(\pi z^2)}}{4(\pi)^{3/2} z^3}+O\left(\frac{1}{z^4}\right)  , &   \text{if} \hspace{.12cm} -\pi +k\pi <\text{arg}(z)< -\frac{\pi}{2}+k\pi \\
\end{cases}
,
\label{eqn:3.39}
\end{equation}
and
\begin{equation}
S(z)\sim  \begin{cases} \frac{\sqrt{2}}{4}+\sqrt{\frac{2}{\pi}}\frac{1}{8z^2}+\frac{\sin{(\pi z^2)}}{2\pi z}+\frac{\cos{(\pi z^2)}}{4(\pi)^{3/2} z^3}+O\left(\frac{1}{z^4}\right) , & \text{if} \hspace{.12cm} -\frac{\pi}{2}+k\pi <\text{arg}(z)< \frac{\pi}{2}+k\pi  \\
-\frac{\sqrt{2}}{4}-\sqrt{\frac{2}{\pi}}\frac{1}{8z^2}+\frac{\sin{(\pi z^2)}}{2\pi z}+\frac{\cos{(\pi z^2)}}{4(\pi)^{3/2} z^3}+O\left(\frac{1}{z^4}\right) , &   \text{if} \hspace{.12cm} -\pi +k\pi <\text{arg}(z)< -\frac{\pi}{2}+k\pi \\
\end{cases}
,
\label{eqn:3.40}
\end{equation}
where $k=0,1,2,\cdot\cdot\cdot.$
\label{thm:2}
\end{thm}

\begin{proof}
Let us first assume that for $|z|\gg1$, Fresnel's functions are approximated by $C(z)\sim\tilde{C}(z)$ and $S(z)\sim\tilde{S}(z)$. Then using  (\ref{eqn:2.16}), (\ref{eqn:2.17}) and (\ref{eqn:2.18}) yields
\begin{align}
\int\limits_0^z e^{i\pi\eta^2}d\eta&=\frac{e^{{i\pi}{z^2}}}{\sqrt{i\pi}}\text{daw}\left({\sqrt{i\pi}}{z}\right)
=z\hspace{.05cm}{}_1F_1\left(\frac{1}{2};\frac{3}{2};i\pi z^2\right)\label{eqn:3.41}
\\ &=\int\limits_0^z \cos{(\pi\eta^2)}d\eta +i \int\limits_0^z \sin{(\pi\eta^2)}d\eta \label{eqn:3.42}
\\ &=z\hspace{.05cm}{}_1F_2\left(\frac{1}{4};\frac{1}{2},\frac{5}{4};-\frac{\pi z^2}{4}\right)+i \frac{\pi z^3}{3}\hspace{.05cm}{}_1F_2\left(\frac{3}{4};\frac{3}{2},\frac{7}{4};-\frac{\pi z^2}{4}\right) \label{eqn:3.43}
\\ &=C(z)+i S(z)\sim \tilde{C}(z)+i \tilde{S}(z), \hspace{.12cm} |z|\gg1. \label{eqn:3.44}
\end{align}
And we observe that the asymptotic expansion of (\ref{eqn:3.41}) is a sum of two parts $\tilde{C}(z)$ and $\tilde{S}(z)$.

Now, setting $\xi=i\pi z^2, a=\frac{1}{\alpha}=\frac{1}{2}$ and  $b=\frac{1}{\alpha}+1=\frac{3}{2}$ in (\ref{eqn:3.24}), and taking account that $\alpha=2$ is even yields
\begin{multline}
\frac{{}_1F_1\left(\frac{1}{2};\frac{3}{2};i\pi z^2\right)}{\Gamma\left(\frac{3}{2}\right)}=\frac{e^{\pm i\frac{\pi}{2} }}{(i\pi z^2)^{\frac{1}{2}}}\left\{\sum\limits_{n=0}^{R-1}\frac{\left(\frac{1}{2}\right)_n}{n!}(-i\pi z^2)^{-n}+O(|z^2|^{-R})\right\}\\
+\frac{1}{\Gamma\left(\frac{1}{2}\right)}\frac{e^{i\pi z^2}}{i\pi z^2}\left\{\sum\limits_{n=0}^{S-1}\frac{(1)_n \left(\frac{1}{2}\right)_n}{n!}(\xi)^{-n}+O(|\xi|^{-S})\right\}.
\label{eqn:3.45}
\end{multline}
Moreover, $$\xi=i\pi z^2\Rightarrow z=\left(\frac{\xi}{i\pi}\right)^\frac{1}{2}=\left(\frac{|\xi|}{\pi}\right)^\frac{1}{2}e^{i\left(\frac{\text{arg}(\xi)}{2}-\frac{\pi}{4}\right)}$$
that gives
\begin{equation}
\text{arg}(z)=\frac{\text{arg}(\xi)}{2}-\frac{\pi}{4}+k\pi, k=0,1,2,\cdot\cdot\cdot. \label{eqn:3.46}
\end{equation}

Rearranging terms on one hand, while neglecting higher order terms on another hand, yields
\begin{multline}
\int\limits_0^z e^{i\pi\eta^2}d\eta=z\hspace{.05cm}{}_1F_1\left(\frac{1}{2};\frac{3}{2};i\pi z^2\right)=\pm\frac{\sqrt{2}}{4}\mp\sqrt{\frac{2}{\pi}}\frac{1}{8z^2}-\frac{\sin{(\pi z^2)}}{2\pi z}-\frac{\cos{(\pi z^2)}}{4(\pi)^{3/2} z^3}+O\left(\frac{1}{|z|^{4}}\right)\\
+i\left[\pm\frac{\sqrt{2}}{4}\pm\sqrt{\frac{2}{\pi}}\frac{1}{8z^2}+\frac{\sin{(\pi z^2)}}{2\pi z}+\frac{\cos{(\pi z^2)}}{4(\pi)^{3/2} z^3}+O\left(\frac{1}{|z|^{4}}\right)\right].
\label{eqn:3.47}
\end{multline}
where the positive sign is now taken if

\begin{equation}
-\frac{\pi}{2}+k\pi <\text{arg}(z)< \frac{\pi}{2}+k\pi , \label{eqn:3.48}
\end{equation}
and the negative sign if

\begin{equation}
-\pi +k\pi <\text{arg}(z)< -\frac{\pi}{2}+k\pi.
\label{eqn:3.49}
\end{equation}
Hence, comparing (\ref{eqn:3.44}) with (\ref{eqn:3.47}) gives (\ref{eqn:3.39}) and (\ref{eqn:3.40}).
\end{proof}

\subsection{Asymptotic evaluation of Gordeyev's integral}\label{subsec:3.3}
In this section, we derive the asymptotic expansion of Gordeyev's integral using the asymptotic expansion of the confluent hypergeometric function ${}_1F_1$ and present the results in Theorem \ref{thm:3}.

We first observe that for complex $\omega=\omega_r+i\omega_i$ and complex $\nu=\nu_r+i\nu_i$, if we set $\sqrt{\nu_r+i\nu_i}=\tilde{\nu}_r+i\tilde{\nu}_i$, where the subscripts $r$ and $i$ stand for real and imaginary parts respectively, then
\begin{equation}
\frac{\omega-n}{\sqrt{2\nu}}=\frac{(\omega_r+n)\tilde{\nu}_i+\omega_i\tilde{\nu}_i}{\sqrt{2}(\tilde{\nu}_r^2
+\tilde{\nu}_i^2)}+i\frac{\omega_i\tilde{\nu}_r-(\omega_r+n)\tilde{\nu}_i}{\sqrt{2}(\tilde{\nu}_r^2+\tilde{\nu}_i^2)}.
\label{eqn:3.50}
\end{equation}
And so
\begin{equation}
\text{arg}\left(\frac{\omega-n}{\sqrt{2\nu}}\right)
=\arctan\left[\frac{\omega_i\tilde{\nu}_r-(\omega_r+n)\tilde{\nu}_i}{(\omega_r+n)\tilde{\nu}_i+\omega_i\tilde{\nu}_i}\right].
\label{eqn:3.51}
\end{equation}

\begin{thm}
Let $\lambda=\lambda_r+i\lambda_i$, $\omega=\omega_r+i\omega_i$, $\nu=\nu_r+i\nu_i$ and $\sqrt{\nu}=\sqrt{\nu_r+i\nu_i}=\tilde{\nu}_r+i\tilde{\nu}_i$, where the subscripts $r$ and $i$ stand for real and imaginary parts respectively, and let \begin{equation}
\theta=\text{arg}\left(\frac{\omega-n}{\sqrt{2\nu}}\right)
=\arctan\left[\frac{\omega_i\tilde{\nu}_r-(\omega_r+n)\tilde{\nu}_i}{(\omega_r+n)\tilde{\nu}_i+\omega_i\tilde{\nu}_i}\right].
\label{eqn:3.52}
\end{equation}

\begin{enumerate}
\item  Then for any fixed $\nu$ and any fixed $\omega$, Gordeyev's integral $G_\nu(\omega,\lambda)$ is asymptotically given by
\begin{equation}
G_\nu(\omega,\lambda) \sim\frac{-i\omega}{2\sqrt{\pi\nu\lambda}}\sum\limits_{n=-\infty}^{\infty}\left[1-\frac{4n^2-1}{8\lambda}+O\left(\frac{1}{\lambda^2}\right)\right]
Z\left(\frac{\omega-n}{\sqrt{2\nu}}\right),
|\lambda|\gg 1,
\label{eqn:3.53}
\end{equation}
where $Z$ is the plasma dispersion (Fried-Conte) function, see equation (\ref{eqn:1.4}), and $-\frac{\pi}{2}<\text{arg}(\lambda)<\frac{\pi}{2}$.
\item For any fixed $\lambda$, if $|\omega|\gg1 $ and $\nu$ is fixed, or if $|\nu|\rightarrow 0$ and $\omega$ is fixed, then
\begin{multline}
G_\nu(\omega,\lambda)\sim \frac{-i\omega}{\sqrt{2\nu}}e^{-\lambda}\sum\limits_{n=-\infty}^{\infty}I_n(\lambda) \Bigl\{ie^{-\left(\frac{\omega-n}{\sqrt{2\nu}}\right)^2}\left[-\frac{\pi}{2}\left(\frac{\sqrt{2\nu}}{\omega-n}\right)^2+O\left(\frac{\sqrt{2\nu}}{\omega-n}\right)^4\right]
\\-\frac{1}{2}\left(\frac{\sqrt{2\nu}}{\omega-n}\right)-\frac{\sqrt{\pi}}{4}\left(\frac{\sqrt{2\nu}}{\omega-n}\right)^3+O\left(\frac{\sqrt{2\nu}}{\omega-n}\right)^5\Bigr\},
\label{eqn:3.54}
\end{multline}
if $-\frac{\pi}{4}+k\pi <\theta< \frac{3\pi}{4}+k\pi, \hspace{.12cm} k=0,1,2,\cdot\cdot\cdot$. And
\begin{multline}
G_\nu(\omega,\lambda)\sim \frac{-i\omega}{\sqrt{2\nu}}e^{-\lambda}\sum\limits_{n=-\infty}^{\infty}I_n(\lambda) \Bigl\{i\sqrt{\pi}e^{\left(\frac{\omega-n}{\sqrt{2\nu}}\right)^2}\left[2+\frac{\sqrt{\pi}}{2}\left(\frac{\sqrt{2\nu}}{\omega-n}\right)^2+O\left(\frac{\sqrt{2\nu}}{\omega-n}\right)^4\right]
\\-\frac{1}{2}\left(\frac{\sqrt{2\nu}}{\omega-n}\right)-\frac{\sqrt{\pi}}{4}\left(\frac{\sqrt{2\nu}}{\omega-n}\right)^3+O\left(\frac{\sqrt{2\nu}}{\omega-n}\right)^5\Bigr\},
\label{eqn:3.55}
\end{multline}
if $-\frac{3\pi}{4}+k\pi <\theta< -\frac{\pi}{4}+k\pi$.
\item If $|\lambda|\gg1$ and $-\frac{\pi}{2}<\text{arg}(\lambda)<\frac{\pi}{2}$, and $|\omega|\gg1 $ while $\nu$ is fixed, or and $|\nu|\rightarrow 0$ while $\omega$ is fixed, then
\begin{multline}
G_\nu(\omega,\lambda)\sim \frac{-i\omega}{2\sqrt{\pi\nu\lambda}}\sum\limits_{n=-\infty}^{\infty} \Bigl\{ie^{-\left(\frac{\omega-n}{\sqrt{2\nu}}\right)^2}\Bigl[-\frac{\pi}{2}\left(\frac{\sqrt{2\nu}}{\omega-n}\right)^2
+\frac{(4n^2-1)\pi}{16\lambda}\left(\frac{\sqrt{2\nu}}{\omega-n}\right)^2\\+O\left(\frac{\sqrt{2\nu}}{\lambda(\omega-n)}\right)^2\Bigr]
-\frac{1}{2}\left(\frac{\sqrt{2\nu}}{\omega-n}\right)
+\frac{(4n^2-1)}{16\lambda}\left(\frac{\sqrt{2\nu}}{\omega-n}\right)-\frac{\sqrt{\pi}}{4}\left(\frac{\sqrt{2\nu}}{\omega-n}\right)^3
\\+\frac{(4n^2-1)\sqrt{\pi}}{32\lambda}\left(\frac{\sqrt{2\nu}}{\omega-n}\right)^3
+O\left(\frac{\sqrt{2\nu}}{\lambda^2(\omega-n)}\right)\Bigr\},\hspace{2.1cm}
\label{eqn:3.56}
\end{multline}
if $-\frac{\pi}{4}+k\pi <\theta< \frac{3\pi}{4}+k\pi, \hspace{.12cm} k=0,1,2,\cdot\cdot\cdot$. And
\begin{multline}
G_\nu(\omega,\lambda)\sim \frac{-i\omega}{2\sqrt{\pi\nu\lambda}}\sum\limits_{n=-\infty}^{\infty} \Bigl\{i\sqrt{\pi}e^{-\left(\frac{\omega-n}{\sqrt{2\nu}}\right)^2}\Bigl[2-\frac{(4n^2-1)}{4\lambda}+\frac{\sqrt{\pi}}{2}\left(\frac{\sqrt{2\nu}}{\omega-n}\right)^2
\\-\frac{(4n^2-1)\sqrt{\pi}}{16\lambda}\left(\frac{\sqrt{2\nu}}{\omega-n}\right)^2+O\left(\frac{\sqrt{2\nu}}{\lambda(\omega-n)}\right)^2\Bigr]
-\frac{1}{2}\left(\frac{\sqrt{2\nu}}{\omega-n}\right)
+\frac{(4n^2-1)}{16\lambda}\left(\frac{\sqrt{2\nu}}{\omega-n}\right)\\-\frac{\sqrt{\pi}}{4}\left(\frac{\sqrt{2\nu}}{\omega-n}\right)^3
+\frac{(4n^2-1)\sqrt{\pi}}{32\lambda}\left(\frac{\sqrt{2\nu}}{\omega-n}\right)^3
+O\left(\frac{\sqrt{2\nu}}{\lambda^2(\omega-n)}\right)\Bigr\},\hspace{0cm}
\label{eqn:3.57}
\end{multline}
if $-\frac{3\pi}{4}+k\pi <\theta< -\frac{\pi}{4}+k\pi, \hspace{.12cm} k=0,1,2,\cdot\cdot\cdot$.
\end{enumerate}
\label{thm:3}
\end{thm}

\begin{proof}
\begin{enumerate}
\item For $|\lambda|\gg1$, we have using formula 9.7.1 in \cite{AS} that
\begin{equation}
I_n(\lambda)=\frac{e^\lambda}{\sqrt{2\pi\lambda}}\left[1-\frac{4n^2-1}{8\lambda}
+O\left(\frac{1}{\lambda^2}\right)\right],\hspace{.12cm}-\frac{\pi}{2}<\text{arg}(\lambda)<\frac{\pi}{2}.
\label{eqn:3.58}
\end{equation}
Substituting in (\ref{eqn:2.19}) gives (\ref{eqn:3.53}).
\item Setting $z=\frac{\omega-n}{\sqrt{2\nu}}$ in (\ref{eqn:3.34a}) and (\ref{eqn:3.34b}), and letting $|\omega|\gg1 $ while $\nu$ is fixed, or and $|\nu|\rightarrow 0$ while $\omega$ is fixed, yields respectively
\begin{multline}
Z\left(\frac{\omega-n}{\sqrt{2\nu}}\right)\sim ie^{-\left(\frac{\omega-n}{\sqrt{2\nu}}\right)^2}\left[-\frac{\pi}{2}\left(\frac{\sqrt{2\nu}}{\omega-n}\right)^2+O\left(\frac{\sqrt{2\nu}}{\omega-n}\right)^4\right]
\\-\frac{1}{2}\left(\frac{\sqrt{2\nu}}{\omega-n}\right)-\frac{\sqrt{\pi}}{4}\left(\frac{\sqrt{2\nu}}{\omega-n}\right)^3+O\left(\frac{\sqrt{2\nu}}{\omega-n}\right)^5,
\label{eqn:3.59}
\end{multline}
if $-\frac{\pi}{4}+k\pi <\theta=\text{arg}\left(\frac{\omega-n}{\sqrt{2\nu}}\right)< \frac{3\pi}{4}+k\pi, \hspace{.12cm} k=0,1,2,\cdot\cdot\cdot$. And
\begin{multline}
Z\left(\frac{\omega-n}{\sqrt{2\nu}}\right)\sim i\sqrt{\pi}e^{\left(\frac{\omega-n}{\sqrt{2\nu}}\right)^2}\left[2+\frac{\sqrt{\pi}}{2}\left(\frac{\sqrt{2\nu}}{\omega-n}\right)^2+O\left(\frac{\sqrt{2\nu}}{\omega-n}\right)^4\right]
\\-\frac{1}{2}\left(\frac{\sqrt{2\nu}}{\omega-n}\right)-\frac{\sqrt{\pi}}{4}\left(\frac{\sqrt{2\nu}}{\omega-n}\right)^3+O\left(\frac{\sqrt{2\nu}}{\omega-n}\right)^5,
\label{eqn:3.60}
\end{multline}
if $-\frac{3\pi}{4}+k\pi <\theta=\text{arg}\left(\frac{\omega-n}{\sqrt{2\nu}}\right)< -\frac{\pi}{4}+k\pi, \hspace{.12cm} k=0,1,2,\cdot\cdot\cdot$. Hence, substituting (\ref{eqn:3.59}) and (\ref{eqn:3.60}) in (\ref{eqn:2.20}) respectively gives (\ref{eqn:3.54}) and (\ref{eqn:3.55}).
\item On the other hand, combining (\ref{eqn:3.54}) and (\ref{eqn:3.55}) with (\ref{eqn:3.53}) respectively gives (\ref{eqn:3.56}) and (\ref{eqn:3.57}).
\end{enumerate}
\end{proof}
\section{Discussion and conclusions}\label{sec:4}

Having evaluated Dawson's integral in terms of the confluent hypergeometric function, other related functions including the complex error (Faddeeva's integral), Fried-Conte (plasma dispersion) function, (Jackson) function, Fresnel functions and Gordeyev's integral were also evaluated in terms of the confluent hypergeometric function .

Using the asymptotic expansions of the confluent hypergeometric function, the asymptotic expansion for $|z|\gg 1$ of Dawson's integral were derived and consequently the asymptotic expansions of the complex error function (Faddeeva's integral), Fried-Conte (plasma dispersion) function, (Jackson) function, Fresnel functions and Gordeyev's integral were evaluated (Theorem \ref{thm:1}, Theorem \ref{thm:2} and Theorem \ref{thm:3}). To obtain, on the other hand, the asymptotic expansion of these functions for small arguments $|z|\ll1$ one should keep the first few terms in the Taylor series representing the confluent hypergeometric function.

It is also important to point that asymptotic expansions of Gordeyev's integral that takes into account the properties of an electromagnetic wave propagating in a hot plasma, which are the wave frequency, the perpendicular and parallel components of the wave vector, were carefully derived (Theorem \ref{thm:3}).

Moreover, writing these functions in terms confluent hypergeometric function confirms once again that these functions are entire on the whole complex plane $\mathbb{C}$ since the confluent hypergeometric function is entire on the whole complex plane $\mathbb{C}$.


\end{document}